\documentclass[11 pt]{amsart}
\usepackage[utf8]{inputenc}
\usepackage{mathtools}

\makeatletter
\@namedef{subjclassname@2020}{\textup{2020} Mathematics Subject Classification}
\makeatother

\usepackage{cite}

\allowdisplaybreaks[1]

\usepackage{amsmath,amsthm,amssymb}
\usepackage{hyperref}
\usepackage{amsfonts}
\usepackage{amsmath}
\usepackage{mathrsfs}
\usepackage{tabto}
\usepackage{changepage} 
\usepackage{fullpage} 
\usepackage{graphicx}
\usepackage{enumitem} 

\newcommand{\ity}{\infty}
\newcommand{\C}{\mathbb{C}}

\newcommand{\N}{\mathbb{N}}
\newcommand{\Z}{\mathbb{Z}}

\numberwithin{equation}{section}
\newtheorem{theorem}{Theorem}[section]
\newtheorem{lemma}[theorem]{Lemma}
\newtheorem{corollary}[theorem]{Corollary}

\newtheorem{remark}[theorem]{Remark}
\newtheorem{example}[theorem]{Example}

\title [dynamics of permutable entire functions]{A survey on permutable transcendental entire functions and their dynamics}
\author{ DINESH KUMAR}
\date{}
\thanks {The research work of the first  author is supported by ANRF(SERB) research grant TAR/2023/000197}
\begin{document}
\maketitle


\begin{abstract}
\noindent In this  paper, we survey  important results on the dynamics of permutable transcendental entire functions  from 1958 to 2025. We have discussed the forms of  transcendental entire functions that could be permutable. We have also discussed  the dynamical properties of permutable entire functions. For instance,  how the Fatou and Julia sets are related. We have also considered  the relation of the subsets which are classified based on the dynamics of the orbit of a point in the complex plane. They are the filled  Julia set $K(f)$ (set of points whose orbits remain bounded), the escaping set $I(f)$ (set of points whose orbits escape to infinity) and the bungee set $BU(f)$ (set of points which are neither bounded nor escaping to infinity). The escaping set is further classified based on the speed of escape, namely, the fast escaping set $A(f)$ or other sets such as $I_0(f)$ and $T(f)$, whose definitions are mentioned in this article. We have also discussed the dynamical relations of these sets. The bungee set sometimes contains a wandering domain, and in this article we have discussed the dynamical relations and conditions that dictates whether wandering domains exists or if it is contained within the bungee set.
\end{abstract}


\section{Introduction}
\noindent \large{ A function $f: \mathbb{C} \to \mathbb{C}$ is called entire if it is complex differentiable at all points and it is called transcendental if it has an \emph{essential singularity} at $\ity.$ Polynomials are precisely the entire functions having a \emph{pole} at $\ity.$ Two entire functions $f,\ g$ are said to be permutable  if $f(g(z))=g(f(z))$ for all $z\in\C.$ It is also written as $f \circ g = g \circ f$. Two permutable functions are also said to commute with each other. 	
It is easy to see that  $z^m $ commutes with $z^n$ for every $m,n\in\N$. Two iterates $f^m, f^n$ of every entire function $f$ commute. We refer \cite[Theorem 3.4.4]{prasolov} for other examples. 

\noindent One of the  remarkable families of polynomials are the Chebyshev polynomials $T_n(z)$ which are recursively defined by the relation: $T_{n+1}({z})=2 {z}T_{n}({z})-T_{n-1}({z})$ with $T_0({z})=1, T_1({z})={z}$ \cite[p.\, 100]{prasolov}. It is easy to see that any two Chebyshev polynomials $T_m({z})$ and $T_n({z})$ commute. 
 To define equivalence of two pairs of polynomials, let $P, Q$ be a pair of polynomials.  We  say that the pair of polynomials $\phi \circ P \circ \phi^{-1}$ and $\phi \circ Q \circ \phi^{-1}$ are equivalent to the pair of polynomials $P$ and $Q$ if $\phi({z})= C {z}+D,$  $C, D \in {\C}, C\neq 0.$ It is easy to check that if  the pair of polynomials $P$ and $Q$ commute, then so does its equivalent pair $\phi \circ P \circ \phi^{-1}$ and $\phi \circ Q \circ \phi^{-1}.$ Moreover, the Julia set is same for $P, Q, \phi \circ P \circ \phi^{-1}$ and $\phi \circ Q \circ \phi^{-1}.$ 

\noindent The following theorem of Ritt gives a complete classification of pair of polynomials equivalent to a pair of commuting polynomials $P,$  $Q$.
It says that if two complex polynomials commute, then up to affine conjugation, they are either both monomials or both Chebyshev polynomials. 
\begin{theorem}\cite[Theorem 3.4.4]{prasolov}
Let there be two permutable polynomials $\mathfrak{f}$ and $\mathfrak{g}$. Then one of the following pairs are equivalent to the pair of $(\mathfrak{f,g})$:
\begin{enumerate}
\item{${z}^m$ and $\epsilon {z}^n$ where $\epsilon^{m-1}=1$}
\item{$\pm T_m({z})$ and $\pm T_n({z}),$ where $T_m({z})$ and $T_n({z})$ are Chebyshev polynomials; }
\item{$\epsilon_1Q^{ k}({z})$ and $\epsilon_2Q^{ l}({z})$ where $\epsilon_1^q=\epsilon_2^q=1$ and $Q({z})={z}P({z}^q)$ and $Q^{ 1}=Q,\ Q^{ 2}=Q \circ Q, \ Q^{ 3}=Q \circ Q \circ Q$ and so on.}
\end{enumerate}
\end{theorem}
\par 
Commutability of two entire functions $f$ and $g$ gives that $f\circ g^2\circ f=g\circ f^2\circ g,$ and in general, $f\circ g^n\circ f^{n-1}=g\circ f^n\circ g^{n-1}$  for all $n.$ This indicates a possible relation between the iterative behaviour of  $f$ and $g.$	In order to make it precise, we need two definitions. The Fatou set of $f$, denoted by  $F(f)$  is defined as the maximal open subset of $\C$ in which the family of iterates $\{f^n\}_{n>0}$ is normal. The Julia set, denoted by $J(f)$ is the complement of the Fatou set. A point $z_0$ is called a periodic point of $f$ if $f^m(z_0)=z_0$ for some positive integer $m.$ The smallest such $m$ is called the period of $z_0$. The periodic points of period $1$ are called the  fixed points. If $z_0$ is a periodic point of period $m$, then $(f^{(m)})'(z_0)$ is called the multiplier of $z_0.$
A periodic point $z_0$ is called  attracting, repelling or indifferent (neutral) according as the modulus of its multiplier is less than, greater than, or equal to $1$ respectively. The indifferent periodic point is called rationally indifferent or parabolic if its multiplier is a root of unity. Otherwise, it is called irrationally indifferent periodic point.

 The Julia set is also the closure of repelling periodic points \cite{Baker_permu2}. 
 Furthermore, the sets $J(f)$ and $F(f)$ are completely invariant  under $f.$
 It is Julia who first revealed the dynamical relevance of permutability by proving that if rational functions $f, \, g$ each with degree at least two are permutable then $J(f)=J(g)$ \cite[Theorem 4.2.9]{beardon}. In particular, this result is true when both $f$ and $g$ are polynomials.
\par  A polynomial commuting with a transcendental entire function is necessarily of the form $P_1(z)=  e^{{2s\pi i}/p}\, z+d$ for some natural numbers $s, p $ and a complex number $d.$ 
  This was proved by  Baker \cite{Baker_german} and, also independently by  Iyer \cite{iyer}.  This article deals with permutable entire functions when  at least one of them  is transcendental. The forms and properties of entire functions commuting with a given function are surveyed in section 2. Then the relation between the dynamics (Fatou and Julia set) of commuting functions are presented. \\
\noindent\large{Given an entire function, all the possibilities of an entire function  commuting with it are found. Most of the results surveyed in Section 2 considers  functions  listed below:
\begin{enumerate}
\item $ce^{dz}+b\, (cd\neq 0, c, d, b\in\C);$
\item $z+Ce^{az}$, where  $aC\neq 0;$
\item $Q+Re^P$, where $P, Q$ and $R$ are polynomials satisfying some conditions;
\item $f(e^z) + g(e^{-z}),$ where $f, g$ are entire and $f$ is periodic of finite order;
\item $P(e^z) + Q(e^{-z}),$ where $P$ and $Q$ are polynomials at least one of which is not constant;
\item $P(z)e^{Q(z)}, P, Q$ non-constant polynomials and $P(0)\neq 0$;
\item $\cos\sqrt z;$
\item $\sin z +Q(z), Q(z)$ is a polynomial;
\item $\operatorname{sin} P(z),$ where $P(z)$ is a non-constant polynomial.
\end{enumerate}
 In Section 3, we survey various dynamical properties of permutable transcendental entire functions. We discuss for two commuting transcendental entire functions, how various subsets of Fatou and Julia sets are related. The results of Baker have been discussed which enhances our understanding of iterative behaviour of two commuting functions on a Fatou component.  A class of commuting functions has been explored for which Fatou components have similar behaviour. We discuss several conditions implying equality of Fatou set (respectively Julia set) of two commuting entire functions. Moreover, the results also specify the form of the commuting functions.
  Finally, we also consider another dynamical partition of $ \mathbb{C}$ in terms of three disjoint sets namely the bungee set $BU(f),$ the filled Julia set $K(f)$  and  the escaping set  $I(f).$ We also discuss how the bungee set, the filled Julia set and the escaping set of composite entire functions are related. As a special case we observe that for two commuting maps the bungee set, the filled Julia set and the escaping set of the composition is completely invariant under individual functions.}

\section{Various Forms of permutable entire functions}

\large{We start with the exponential function. A result of Baker determines all non-linear entire functions that commute with $e^z.$}
\begin{theorem}\label{bak1} \cite{Baker_german}
Suppose $h$ is a non-linear entire function which is permutable with  the exponential function  $f(z)=ce^{dz}+b\, (cd\neq 0, c, d, b\in\C)$ then $h=f^n$ for some $n\in\N.$
\end{theorem}
\noindent \large{In 1962, Baker established a result regarding the size of the set of functions that  permute with a given entire function.
Before stating it, we denote the collection of all the entire functions $g$ that  permutes with the given function $f$ by $\mathfrak C(f).$} In particular, the iterates of $f$ namely, $f^n \in \mathfrak C(f)$ \cite{Baker_permu}. Using Theorem \ref{bak1}, we conclude that $\mathfrak C(e^z)$ is countably infinite.

\begin{theorem}
({\cite{Baker_permu}, Theorem 1)
Suppose $f$ is a transcendental entire function  or is a polynomial with degree at least two. Also, suppose $f$ has a periodic point  which is either repelling or
of multiplier $1$, then $\mathfrak C(f)$ is countably infinite. }
\end{theorem}
\noindent\large{ Hypothesis of above result is always true for polynomials. 
This is because a rational function always contains at least one fixed point which is either repelling or has multiplier $1$ \cite[p. 168]{fatou}.} Hypothesis of above result is also always true for transcendental entire functions by the result of Baker \cite{Baker_permu}.
\noindent\large{Finding entire maps that commutes with the function $z+e^z$ is not easy using the previous results. Using tools from Nevanlinna theory, Kobayashi in 1980 determined all entire functions with finite order that commutes with functions of the form $f(z)=z+Ce^{az}$ where $C\neq 0, \, a$ are complex numbers. The order $\rho(f)$ of an entire function $f$ is defined as:
\[\rho(f)=\lim\underset{r\rightarrow \infty}{\sup}\frac{\log\, \log \, M(r, f)}{\log \, r}\] 
where $M(r, f)={max}_{|z| = r } |f(z)|.$ 
 It measures the growth of a function as $|z|\to\ity.$ Note that $M(r, e^z)=e^r$ and $\rho(e^z)=1.$}

\begin{theorem} \label{1} (\cite{Kobayshi}, Theorem 1)
\large{Let $f(z)=z+Ce^{az}$, where  $aC\neq 0$. Let $g$ be a non-constant entire function of finite order that is permutable with $f$. Then either $g(z)=f(z)+D$ or $g(z)=z+D$ where $D$ is a constant with $e^{aD}$ =1.
}
\end{theorem}
\noindent

\noindent Observe that $f$ in Theorem \ref{1} satisfies $f''(z)-a^2 f(z) +a^2 z=0.$
 Also, $g(z)=f(z)+D$  satisfies the differential equation $g''(z)-a^2 g(z)+a^2(z+D)=0.$
 Motivated by this, Zheng and Zhou considered entire functions satisfying certain differential equations.  In order to state their result,  we need a definition. The lower order of a function $f$    is defined as
\[\rho(f)=\lim\underset{r\rightarrow \infty}{\inf}\frac{\log\, \log \, M(r, f)}{\log \, r}.\] 

\begin{theorem} \label{Diff_Zheng} (\cite{Zheng_Zhou}, Theorem 1)
\large{Let $f$ and $g$ be two permutable entire functions of finite order and $f$ is of positive lower order. Let $P_i(z) \ (i=0,1, \ldots , n+1$ with $n\geq 1$)  be polynomials not all identically zero. If $f(z)$ satisfies a differential equation 
$P_0(z)f^{(n)}(z)+ \cdots P_n(z)f(z)+P_{n+1}(z)=0,$
then there exist polynomials $Q_0(z), \ldots, Q_{n+1}(z)$ not all identically zero such that
$Q_0(z)g^{(n)}(z)+ \cdots Q_n(z)g(z) +Q_{n+1}(z) =0.$
}
\end{theorem}


\noindent\large{Above theorem basically ensures that permutability and growth restrictions on one of the function which satisfies some non-trivial linear differential equation with polynomial coefficents forces the other function also to satisfy a differential equation of similar vein with possibly different coefficients. For example, consider $f(z)=\sin z$ and $g(z)=-z.$ Then $f$ commutes with $g$. $f$ is of finite lower order.  $f$ satisfies $f''(z)+f=0,$ while $g$ satisfies $g''(z)=0.$ }
\noindent \large{The next theorem,  a generalisation of  Theorem \ref{1} determines all entire maps commuting with $Q+Re^P$, where $Q, R$ and $P$ are polynomials satisfying some conditions.}
\begin{theorem}
\large({\cite{Zheng_Zhou}, Theorem 3)
Let $f=Q+Re^P$, where $Q$ and $R$ are polynomials, $R$ not identically $0$  and $P$ is a non-constant polynomial of degree $n.$ Let $g$ be a non-linear entire function of finite order that commutes with $f$. Then $g=Af+B$, where $A^{n}=1, \ B= \frac{p_{n-1}(A-1)}{np_n},$ and  $\ p_n, \ p_{n-1}$ are coefficients of the leading and second leading coefficient of $P,$ respectively. 
}
\end{theorem}
As an illustration of above theorem, we consider $f(z)=ze^{z^2}$.  Here $A=-1$ and $B=0$  \, as $p_{n-1}=0. \,$  This implies $g(z)=-ze^{z^2}.$  $\rho(g)=2.$ It is easy to see that $f\circ g=g\circ f=-ze^{z^2}e^{{z^2}e^{2z^2}}.$ Also, if we consider $f(z)=ze^{z^4},$ then $A^4=1$ gives non-trivial choice for $A$ to be $-1, i$ and $-i,$ $B=0$. One can check easily that each of the three choices of $g$ namely, $-f, i\,f $ and $-i\,f$ commutes with $f.$\\ We now show that not all non-trivial choices for $A$ satisfying $A^n=1$ can give different choices of $g$ that commutes with  $f$ as in above theorem. For this, consider $f(z)=z^3\, e^{z^4}.$ One can easily check that out of the three non-trivial possibilites  $-1, i$ and $-i$ of $A,$ only the possibility $g=-f$ commutes with $f.$\\
  More generally, one can take $f(z)=z^{2n+1}e^{z^{2m}}$ where $n, m \geq 1.$ Here, $A^{2m}=1$ and $B=0.$ One can easily check that  $g= -f$ commutes with $f$ because $f\circ g(z)=-f^2(z),$ and $g\circ f(z)=-f^2(z).$\\ Note that, all the functions considered in above examples were odd functions. Infact, we have the following observation:

\begin{remark}\label{remark1}
If $f$ is an odd function, then $f$ commutes with $-f,$  for $f(-f(z))=-f(f(z))=-f^2(z).$ We shall use this fact at several places in this paper.
\end{remark}

\noindent\large{The following theorem determines all entire functions that commutes with $\operatorname{sin} P(z)$ for every polynomial $P.$ }
\begin{theorem}
\large{(\cite{Zheng_Zhou}, Theorem 4)
Let $f(z)=\operatorname{sin} P(z)$ where $P$ is a non-constant polynomial. Let $g$ be a non-linear, finite order, entire function commuting with $f.$ Then, 
\begin{enumerate}[label=(\roman{*})]
\item{if $f(z)$ is an odd function, then $g(z)=f(z)$ or $g(z)=-f(z);$}
\item{if $f(z)$ is not an odd function, then  

	\begin{enumerate}
\item{when $\operatorname{deg}(P)=1,\ g(z)=f(z),$} 
\item{when  $\operatorname{deg}(P)>1$, either $g(z)=f(z)$ or $g(z)=-\operatorname{sin}(P(z)-c),$ where $c$ satisfies $e^{inc}=(-1)^{n+1}, \ n=\operatorname{deg}(P)$ and $P(\operatorname{sin}(-z+c))=P(\operatorname{sin} z)+c$.}
	\end{enumerate}
}
\end{enumerate}
}
\end{theorem}
\noindent\large{As an illustration of the first situation in above theorem, we consider $f(z)=\sin z^3.$ Using Remark \eqref{remark1}, $g(z)=-\sin z^3$ commutes with $f.$}

%

\noindent \large{The forthcoming results proved by Zheng and Yang during 1987-1990, discuss  permutable entire functions that are periodic.}

\begin{theorem} \label{per_1}
\large{(\cite{Zheng}, Theorem 1)
Let $f$ and $g$ be two entire functions of order less than $\frac{1}{2}$ and $h$ be entire. If $h$ is permutable with $f(e^{z})+g(e^{-z}),$ then 
$h(z)=(\frac{k}{s})z + l(z)$, where both $s$ and $k$ are integers with $s>0$ and $l(z)$ is periodic entire function with period $2s\pi i.$
}
\end{theorem}
The following result which is a  consequence of above theorem provides a complete set of entire functions that are permutable with a periodic entire function of finite order.

\begin{theorem}
\large{(\cite{Zheng}, Corollary)
Let $f(z)$ be a periodic entire function of finite order and let
$g(z)$ be entire, permutable with $f(z)$. Then  $g(z)=(\frac{k}{s})z + l(z)$, where both $s$ and $k$ are integers with $s>0$ and $l(z)$ is periodic with period $2s\pi i.$
}
\end{theorem}
We give the  proof of above theorem:
\begin{proof}
Suppose the period of $f$ is $c\neq 0.$ Then, by a change of variable, we may assume without loss of generality that the period is $2\pi i$. Thus, the function $f$ can be written as $f(z)=F(e^z)+G(e^{-z}),$ where $F$ and $G$ are entire functions. In order to apply Theorem \ref{per_1}, we must show that $F$ and $G$ are of zero order (and hence will have order less than $\frac{1}{2}.$) Now, $M(r, f)=M(r, F(e^z)+G(e^{-z}))\geq M(r, F(e^z))-M(1, G(e^{-z}))$ for $r$ sufficiently large. By Polya's result \cite{polya}, we conclude that $F$ is of zero order. Arguing on similar lines, $G$ is also of zero order. Thus, hypothesis of Theorem \ref{per_1} is satisfied. Hence, the function $g$ has the required form.
\end{proof}

\noindent Zheng suspected $k$ to be zero in above theorem and posed a question whether this is always the case or not.\\
The following result by Zheng is a special case of Theorem \ref{per_1} when $f$ and $g$ are polynomials.
\begin{theorem}
\large{(\cite{Zheng}, Theorem 2)
Let P(z) and Q(z)  be polynomials at least one of which is not a constant and $f(z)=P(e^z) + Q(e^{-z}).$ Let $g$ be any non-linear entire function of finite order commuting with $f.$ 
\begin{enumerate}
\item[(i)] {If $P(z) \not \equiv -Q(z) $, then $g=f;$}
\item[(ii)] {If $P(z) = -Q(z),$ for all $z$ then $g=f$ or $g=-f$.}
\end{enumerate}
}
\end{theorem}
\noindent\large{It says that for any two polynomials $P, Q,$ the only entire function of finite order that commutes with $f(z)=P(e^z) + Q(e^{-z})$ is $f$ itself or $-f.$}

\begin{remark}
If we relax the condition of finite order on function $g$ in above theorem, then conclusion may not hold. For instance, if  $P(z)=z$ and $Q(z)=0,$ then $g(z)=e^{e^z}$ commutes with $f(z)=e^z$ and is of infinite order. Here, neither $g=f$ nor $g=-f.$
\end{remark}
As a continuation of Theorem \ref{per_1},  some more results were 
obtained on the permutability of periodic entire functions. Some  relationships between two commuting entire functions was observed. 
Also, some necessary conditions under which a function  permutable with a
given function of the form $G(e^z),$ where $G$ is an entire function was obtained in \cite{Zheng_Yang}.


\begin{theorem}\label{th1}
\large{(\cite{Zheng_Yang}, Theorem 1)
Let $f$ be entire, with order less than $\frac{1}{2}$ and $h$ be entire commuting with $f(e^z).$ If $h(f(0))\neq f(0),$ then h is periodic and can be expressed as $h(z)=g(e^{\frac{z}{s}})$ for some entire function $g$ and positive integer  $s.$
}
\end{theorem}

\begin{theorem}\label{th2}
\large{(\cite{Zheng_Yang}, Theorem 2)
Let $f(z)=P(z)e^{Q(z)}$ with both $P(z), \ Q(z)$ being non-constant polynomials and $P(0)\neq 0.$ Assume that $h(z)$ is entire and permutable with $f(e^{z})$. If $h(f(0))\neq f(0)$, then h(z) is periodic and has the form $h(z)=g(e^{\frac{z}{s}})$ for some entire function $g(z)$ and positive integer $s.$
}
\end{theorem}
Note that, $f$ in Theorem \ref{th1} has order less than $\frac{1}{2},$ while $f$ in Theorem \ref{th2} has order at least $1.$ However, the conclusions of both the theorems are the same.
 The next result provides functions of finite order that are permutable with $\cos\sqrt{z}$.
\begin{theorem}
\large{(\cite{Zheng_Yang}, Theorem 3)
Let $h$ be an entire function of finite order, commuting with $cos\sqrt{z}$. Then either $h(z) = cos\sqrt{z}$ or $h(z) = cz, c \neq 0$.
}
\end{theorem}
Zheng and Yang also proved the following important result using Wiman-Valiron theory.
\begin{theorem}\label{th3}
\large{(\cite{Zheng_Yang1}, Theorem 3)
Let $f(z)=\sin z+Q(z)$, where $Q(z)$ is a polynomial. If $g(z)$  is a non-linear entire function of finite order, which is permutable with $f(z)$ then 
\begin{enumerate}
\item[(i)] {If deg $(Q)=0,$ then $g=f$ or $g=-\sin z+k \pi,$ and $ \, f=\sin z+k \pi, k\in\Z$;}
\item[(ii)] {If deg $(Q)=1,$ then  $g=f+2k\pi$ or $-f+2k\pi$ for some integer $k$;}
\item[(iii)] {If deg $Q>1$ and $Q$ is not odd, then $g=f$;}
\item[(iv)] {If deg $Q>1$ and $Q$ is odd, then $g=f$ or $g=-f$.}
\end{enumerate}
}
\end{theorem} 
\noindent\large{We make some observation from above theorem. The first case gives that $f$ has the form $f(z)=\sin z+c$ for some constant $c.$ On taking $g=-\sin z+c$, the functional equation $f\circ g=g\circ f,$ gives $\sin (-\sin z+c)=-\sin (\sin z+c).$ This holds whenever $c=k\pi, k\in\Z.$\\
For the second case, deg $(Q)=1$ gives $f(z)=\sin z +az+b$ for some $0\neq a, b\in\C.$ On taking $g=f+2k\pi$  for some integer $k$ and solving the functional equation  $f\circ g=g\circ f, $ we get that $a=1.$ That is, $f(z)=\sin z+z+b, \,b\in\C$ commutes with $g=f+2k\pi$ for some integer $k.$ In this case, $f$ satisfies $f(z+2k\pi )=f(z)+2k\pi .$ On the other hand, if we take $g=-f+2k\pi$  for some integer $k$, then on solving the functional equation $f\circ g=g\circ f $ we conclude that $a=1$ and $b=0.$ That is, in this case, $f(z)=\sin z+z$ commutes with $g=-f+2k\pi$ for some integer $k.$ In the third case, as $Q$ is not odd, therefore, $f$ is also not an odd function, and it turns out that the only function commuting with $f$ is $f$ itself. Finally, in the last case, when $Q$ is odd then $f$ is also an odd function and by Remark \eqref{remark1} $g=-f$ commutes with $f.$ }

\section{Dynamical Properties Of Permutable Entire Functions}

\noindent \large{All the entire functions considered in this section are transcendental. Now that we have an idea of how a transcendental entire function permutable with some given function might look, it is natural to discuss the relationship between their Fatou and Julia sets.  In addition, we will see how various subsets of the Fatou and Julia sets are related. This section surveys several results directed towards answering the following question raised by Baker in \cite{Baker_wander}.\\
\textbf{Problem A.}  Suppose $g$ and $h$ are  non-linear entire functions which are permutable. 
Is $J(g) = J(h)? $ }

\noindent\large{We will require some notions which we state here. Montel's theorem  states that if $\mathcal{F}$ is a family of analytic functions defined on a domain $\Omega  \subset \mathbb{C}$ such that each member of this family omits at least  two distinct points in $\C,$ then  $\mathcal{F}$ is normal.
}

\noindent \large{There is a class of functions called Speiser class $\mathcal{S}$ (these functions have a finite number of asymptotic and critical values). Functions in this class are said to be of finite type. There is yet another class of functions called Eremenko-Lyubich class $\mathcal B$ (these functions have a bounded set of asymptotic and critical values). Functions in this class are said to be of bounded type. Clearly, $\mathcal{S}\subset \mathcal{B}.$\\

Suppose $U$ is a  Fatou component of $f.$ Then $U$ is called a wandering domain if all forward iterates are disjoint that is, $f^m (U) \cap f^n (U) = \emptyset$ for $m \neq n.$ If $U$ is not wandering, then $U$ is called a pre-periodic component of $F(f).$  The behavior of $f^n$ on periodic components is well understood.
 We now give the classification of Fatou components of a transcendental entire function $f$ \cite{Yang}.}
\begin{theorem}\label{sec3, thms}\cite{Yang}(Classification of Fatou components)\\

 Suppose $V$ is a Fatou component of an entire function $f$ of  period $m$. Then

\begin{enumerate}
	\item[(i)] $V$ is called an \textbf{attracting domain} if it contains an attracting periodic point $w_0$ of period $m$ and $f^{nm}(z) \to w_0 \quad \text{for } z \in V \text{ as } n \to \infty.$
	\item[(ii)] $V$ is called a \textbf{parabolic domain} if $\partial V$ contains an indifferent periodic point $w_0$ of period $m$ and $f^{nm}(z) \to w_0 \quad \text{for } z \in V \text{ as } n \to \infty.$
	\item[(iii)] $V$ is called a \textbf{Siegel disk} if there exists an analytic homeomorphism $\psi : V \to \mathbb{D},$
	where $\mathbb{D}$ is the unit disk, such that $\psi \circ f^{m} \circ \psi^{-1}(z)=e^{2\pi i \alpha}z$ for some $\alpha \in \mathbb{R}\setminus \mathbb{Q}$.
	\item[(iv)] $V$ is called a \textbf{Baker domain} if there exists $w_0 \in \partial V$ such that $f^{nm}(z) \to w_0 \quad \text{for } z \in V \text{ as } n \to \infty,$ but $f^{m}(w_0)$ is not defined. As $f$ is transcendental, $w_0$ can only be $\infty.$
\end{enumerate}


\end{theorem}
\noindent \large{We now give some results proved by I. N. Baker. These are some initial results that are crucial for understanding the dynamical properties of permutable  entire functions.}
\noindent \large{He proved the following result in 1958. This can be verified easily by taking an  entire function $f(z)$ and $g(z)=f^n(z)$ for any $z.$}
\begin{theorem}
\large{(\cite{Baker_german}, Theorem 7)
Let $f$ and $g$ be permutable entire functions. Then there exists $p \in \mathbb{N}$ and $R > 0$ such that $M(r, g) < M(r, f^p)$ for $r > R $ where $M(r,f)$ denotes the maximum modulus of $f$ in $\{z: |z|=r\}.$
}
\end{theorem}

\noindent\large {The next  result of Baker was foundational in understanding the dynamical analysis of commuting entire functions. It help us understand how the iterative behaviour of one function is related to another function.}

\begin{theorem}\label{sec3,thm1'}
\large{(\cite{Baker_wander}, Lemma 4.4)
If $f$ and $g$ are permutable  entire functions and if $\infty$ is not a limit function of any subsequence of $\{f^n \}_{n\in\N}$ in any component of $F(f)$, nor of a subsequence of $\{g^n\}_{n\in\N}$ in any component of $F(g)$, then $J(f) = J(g).$
} 
\end{theorem}
\noindent\large{ We illustrate above result with an example. Let us consider the  one parameter family of exponential maps,  $h(z) = \mu e^z , \mbox{ for }\mu\in \left(0,\frac{1}{e}\right).$ It is known  that Fatou set of $h,  F(h)$ consists of a simply connected completely invariant component $V$ (which is an attracting basin and the dynamics is attracted to the attracting fixed point contained inside $V$, see \cite{dev}). Using Theorem \ref{bak1}, if $g$ commutes with $h,$ then $g=h^n$ for some $n\in\N.$ Hypothesis of above theorem is clearly satisfied and $F(h)=F(h^n)=F(g).$\\
\noindent For a polynomial $P$ with degree at least two, there is always a Fatou component, namely the basin of attraction of $\infty$ where the iterates $P^n\to\ity$ as $n\to\ity.$ The hypothesis of the above theorem is not true. However, the  Julia set of two permutable polynomials are known to be the same as already discussed in the first section.\\

\noindent\large{Baker \cite[ Lemma 4.5]{Baker_wander}  established that for two permutable entire functions $f$ and $g$ with $g=f+d$ where $d$ is some constant, $J(f)=J(g).$  Poon and Yang in 1998 generalised this as follows:}
\begin{theorem}\label{yang}
\large{(\cite{Yang}, Lemma 2.1)
Suppose $f$ and $g$ are  entire functions such that $g(z)=af(z)+b, \ a,b \in \mathbb{C}, a\neq 0$. If $g$ permutes with $f$, then $J(f)=J(g).$ In addition, $|a|=1.$ 
}
\end{theorem}
\noindent\large{The conclusion of  Theorem \ref{sec3,thm1'} may still hold even if hypothesis is relaxed. For instance, consider $f(z)=z+1+e^{-z}$ and $g(z)=f(z)+2\pi i.$ There is a Baker domain for $f$ containing right half plane (\cite{d1} Example 3.1) } and $F(f)=F(g),$ \cite[Lemma 4.5]{Baker_wander}.

\begin{remark}
Equality of Fatou sets of entire functions $f$ and $g$ may still hold even if they are not permutable. Consider $f(z)= \sin z$ and $g(z)=f(z)+2\pi.$ Note that, $f\circ g \neq g\circ f.$ However, $F(f)=F(g)$ since $g^n(z)=f^n(z)+2\pi$ for each $z.$
\end{remark}

It was shown in (\cite{hua1} Theorem 1) that for two permutable  entire functions in Eremenko-Lyubich class $\mathcal{B}, $ their Fatou sets(respectively Julia sets) are same. In particular, $f, g\in \mathcal S$ commuting, then $J(f)=J(g)$ and hence, $F(f)=F(g).$


\noindent \large{ For two commuting  entire functions $g$ and $h$ with $F(g) = F(h),$ Baker posed a problem that if $V$ is
 a component of $F(g)$ then will $V$ be also a component of $F(h)$ of the same type. This problem has an affirmative answer when $g$ and $h$ are of finite type, see \cite{Ren_Li}. We discuss this in the next result which was proved by Ren and Li in 1997 and which discussed some features of permutable  functions in class $\mathcal{S}$.}

\begin{theorem} \label{S_permu}
\large{\cite{Ren_Li}
Let $f, g \in \mathcal{S}$ and are permutable. Then 

\begin{enumerate}
\item[(i)] {$J(f) = J(g)$}
\item[(ii)] {If $D$ is an (super)attracting domain, a parabolic domain or a Siegel disk of $f$, then $D$ is also a (super)attracting , a parabolic domain or a Siegel disk of $g$, respectively.} 
\end{enumerate}
} 
\end{theorem}
\begin{remark}
We now provide an example of permutable entire functions $g$ and $h$ outside class $\mathcal{S}$ in which Fatou component does not have similar dynamical properties for $g$ and $h.$ 
\end{remark}
\begin{example}
Suppose $g(z)=-\sin z+z, \, h(z)=-\sin z+z+2\pi.$ Clearly, $g\circ h =h\circ g.$
Using \cite[ Lemma 4.5]{Baker_wander} , $F(g)=F(h).$ For \(m\in \mathbb{N}\), $g(2m\pi)=2m\pi$ and $g'(2m\pi)=0.$
Thus, \(\{2m\pi : m\in\mathbb{N}\}\) are attracting fixed points of $g$ for $m \in \N$.
Let \(V_{2m\pi}\) be the attracting domain containing the point \(2m\pi\).
As \(F(g)=F(h)\), \(V_{2m\pi}\) is also a component of \(F(h)\).
Observe that \(V_{2m\pi} \cap V_{2m'\pi}= \phi\) for $m \neq m'$ and
\begin{align*}
	h(2m\pi)
	&=2m\pi+2\pi \\
	& =2(m+1)\pi ,
\end{align*}
which gives
\[
h(V_{2m\pi})=V_{2(m+1)\pi}.
\]
Hence, \(V_{2m\pi}\) is a wandering domain of \(h\).
This shows that the dynamical behaviour of two commuting entire functions can be very different.

\end{example}
\noindent \large{Next theorem was proved by Langley in  1999 which provides a relation between wandering domains of permutable entire functions and equality of their Julia sets.}
\begin{theorem}\cite{Langley} 
Let $f$ and $g$ be permutable entire functions of finite order. If $f$ and $g$ have no wandering domains, then $J(f)=J(g).$ 

\end{theorem}
\noindent\large{Functions of Speiser class $\mathcal{S}$ do not have wandering domains \cite{EL}. There are functions outside Speiser class $\mathcal{S}$ that also do not having wandering domains. For instance, the function $f(z)=z+1+e^{-z}$ is infact, outside $\mathcal{B}$ and has no wandering domains.}
\\
%

\noindent \large{Moving on, the results that were proved by T. W. Ng in 1999 showed that most of the permutable entire functions  have their Julia sets equal. 
The method developed in his paper is useful in solving functional equations.
He provided affirmative answers to Baker's \textbf{Problem A} for a large class of entire functions
including $e^z+P(z)$ and $\sin z+P(z)$ where P is a non-constant polynomial. In fact, he 
proved that for any non-linear entire function $h$ which commutes with $g$ in this class,
$h(z) = cg^n(z) + d$ for some $0\neq c, \, d\in\C.$ Observe that $g^n\circ h=h\circ g^n.$ Using Theorem \ref{yang} we obtain $J(g)=J(g^n)=J(h).$}
\\
\noindent\large{Recall that an entire function $G$ is prime in the entire sense if  $G=g\circ h$ for some entire functions $g, h$ implies either $g$ or $h$ is linear. Also, an entire function $G$ is left-prime in the entire sense if whenever $G=g\circ h$ for some entire functions $g, h$ then  $g$  is linear. The following result of Gross and Yang \cite{Gross1} shows that $e^z+z$ is a prime function.
\begin{theorem}\label{gross-yang}\cite[Corollary]{Gross1}
Suppose $p(z)$ and $q(z)$ are two polynomials satisfying $q(z)\not \equiv 0$ and $p(z)$ is non-constant. Then $q e^z+p(z)$ is prime.
\end{theorem}

\begin{theorem}\label{tw1}
\large({\cite{tw}, Theorem 1)
Let $f$ be an  entire function which satisfies the following conditions:
\begin{enumerate}[label=(\roman{*})]
\item{$f$ is not of the form $H \circ Q$ where $H$ is periodic entire and $Q$ is a polynomial}
\item{$f$ is left prime in entire sense}
\item{$f'$ has at least two distinct zeros}
\item{For each critical value $w,$ the number of points $z$ such that $f(z)=w$ and $f'(z)=0$ is finite.}
\item{The orders of zeros of $f'(z)$ are bounded by $M$ for some $M \in \mathbb{N}$.}
\end{enumerate}
Let $g$ be a non-linear entire function which permutes with $f.$ Then $g(z)=af^n(z)+b,$ where $a^k=1$ for some $k$ and $b \in \mathbb{C}.$ Consequently, $J(f)=J(g).$
} 
\end{theorem}
\noindent\large{The proof of the last part is easy. This is because $J(g)=J(f^n)$ by Theorem \ref{yang} and $J(f^n)=J(f).$}\\
\noindent\large{
 M. Ozawa \large({\cite{ozawa} Theorem 1), showed that if $f$ is of finite order and if for each  $b\in\C,$ the simultaneous equations $f(z)=b, f'(z)=0$ have only a finite number of solutions, then $f$ is left prime in the entire sense. This means that for functions of finite order, the hypothesis (ii) is redundant in Theorem \ref{tw1}.
 We now show that for any non-constant polynomial $P,$ $e^z+P(z)$ and $\sin z+P(z)$ satisfies conditions  (i)-(v). For this we require an application of Borel's Lemma (\cite[Theorem 1.7]{chuang}) to conclude that  the derivatives of $e^z+P(z)$ and $\sin z+P(z)$ have infinitely many zeros. This gives condition (iii). Conditions (ii), (iv) and (v) can be easily checked. It only remains to check condition (i). Suppose on the contrary, $e^z+P(z)$ or $\sin z+P(z)$  is of the form $H\circ Q$ where $H$ is periodic entire function and $Q$ is a polynomial. If $Q$ is linear, then $H\circ Q$ will be periodic. But $e^z+P(z)$ and $\sin z+P(z)$  cannot be  periodic as $P$ is non-constant. Hence, $Q$ has degree greater than $1.$ This implies that order of $H\circ Q$} will also be greater than $1.$ This contradicts the fact that order of $e^z+P(z)$ and $\sin z+P(z)$ is $1.$ Thus, $e^z+P(z)$ and $\sin z+P(z)$ satisfy conditions (i)-(v).
\\
\noindent\large{Observe that, conditions (i)-(v) are satisfied by a large class of entire functions as mentioned below.} 
\begin{theorem}\label{tw2'}
\large{(\cite{tw}, Theorem F)
Let $f$ be an entire function. There exists a countable set $E_f \subset \mathbb{C}$ such that, for any $a \not \in E_f, f_a(z)=f(z)+az$ satisfies the following conditions:
\begin{enumerate}[label=(\roman{*})]
\item{$f_a$ is not periodic}
\item{$f_a$ is prime in the entire sense}
\item{For any $c\in \mathbb{C}$, the simultaneous equations $f_a(z)=c, \ f_a'(z)=0$ have at most one solution}
\item{The orders of zeros of ${f_a}'(z)$ are equal to 1.}
\end{enumerate}
}
\end{theorem}
\noindent\large{We only require to check condition (i) of Theorem \ref{tw1} that $f$ is not of the form $H \circ Q$ where $H$ is periodic entire and $Q$ is a polynomial. This will be done using conditions (i) and (ii) mentioned in above hypothesis. As discussed in above paragraph, $Q$ has degree greater than $1.$ But this means that $f_a$ is not prime in the entire sense which violates condition (ii). Hence, the assertion.}


\noindent\large{Theorem \ref{tw1} applied to $f_a$ given in Theorem \ref{tw2'} gives rise to the following result.}

\begin{theorem}
\large({\cite{tw}, Theorem 2)
Let $f$ be an entire function and define a family of functions $f_a(z)=f(z)+az.$ Then there exists a countable set $E_f \subset \mathbb{C}$ such that for each $a \not \in E_f,$ any non-linear entire function $g$ that permutes with $f_a$ is of the form $g(z)=df^{n}_a(z)+e,$ where $d^k=1$ for some $k\in\N$, and $e \in \mathbb{C}$. Hence, $J(f_a)=J(g)$. 
}
\end{theorem}

\noindent\large{The method developed in \cite{tw} is helpful in solving functional equations. It is also useful in proving the next result.}
\begin{theorem}
\large({\cite{tw}, Theorem 3)
Let q be an entire function (non-constant) and $P$  a polynomial with at least two distinct zeros. Suppose $f$ is of the form $f(z)=P(z)e^{q(z)}$  is prime in the entire sense. Then any non-linear entire function $g$ which commutes with $f$ is of the form $g(z)=cf^n(z)+d$, where $c^k=1$ for some $k\in\N,$ and $d\in \mathbb{C}$. Hence, $J(f)=J(g)$.
} 
\end{theorem}

\noindent \large{It is known that if $q$ is a polynomial and $P$ and $q$ do not have a non-linear common right factor, then $f(z)=P(z)e^{q(z)}$ is prime in the entire sense \cite{tw}.

\noindent \large{The escaping set of an entire function $f$ denoted by $I(f)$ was introduced by Eremenko \cite{Erem}  and  is defined as $I(f)= \{ z \in \mathbb{C} : \ |f^n(z)| \rightarrow \infty \ \text{as} \ n \rightarrow \infty \}$. The three important subsets of $I(f)$ namely,  the fast escaping set $A(f)$ introduced by Bergweiler and Hinkkanen \cite{Bergweiler2} in 1999, $I_0(f)$ and $T(f)$ are defined as follows:\\
\begin{align}
&A(f)=\{z \in \mathbb{C}:\exists \ L \in N \ \text{such that} \hspace{4 pt} |f^{n+L}(z)| > M^n(R,f) \ \text{for all} \ n\geq 0 \} \label{A)}\\
&I_0(f)= \{ z \in \mathbb{C}:  \ \frac{\operatorname{\log}|f^{n+1}(z)|}{\operatorname{log}|f^n(z)|} \rightarrow \infty \ \text{as} \ n \rightarrow \infty \} \label{I_0}\\
&T(f)=\{z \in \mathbb{C}: \frac{\operatorname{\log}\operatorname{\log}|f^{n}(z)|}{n} \rightarrow \infty \ \text{as} \ n \rightarrow \infty \} \label{T}
\end{align}   

\noindent \large{
 We now show that the sets $A(f), I_0(f)$ and $T(f)$ are all subsets of $I(f).$ 
Suppose $z\in A(f).$ In order to show the inclusion $A(f)\subset I(f),$ observe that the iterated maximum modulus $M^n(R,f) $ tends to infinity. This is because as $f$ is transcendental entire, $M(r,f)>r$ for large $r.$ Hence, $|f^n(z)|\to\infty.$\\
Using Wiman and Valiron theory, Eremenko \cite{Erem} proved that for an entire function $h,$ $I(h)\neq \emptyset.$ The proof was based on the behaviour of entire functions near points of maximum modulus. He showed that there exists $z \in I(h)$ such that $|h^{n+1}(z)| \sim M(|h^n(z)|,h)$ as $n \rightarrow \infty$. Also, for entire function $h$,  $\frac{\log M(r,h)}{\log \hspace{2 pt} r} \rightarrow \infty$ as $r \rightarrow \infty$. On taking $r=|h^n(z)|,$ we deduce that $\frac{\log \hspace{2 pt} |h^{n+1}(z)|}{\log \hspace{2 pt} |h^{n}(z)|} \rightarrow \infty$ as $n \rightarrow \infty.$ This shows that $I_0(h) \neq \emptyset$ and $I_{0}(h) \subset I(h).$ 

\noindent Also, for a point $z \in I_0(h),$ it immediately follows from the fact that  $\frac{\log \hspace{2 pt} |h^{n+1}(z)|}{\log \hspace{2 pt} |h^{n}(z)|} \rightarrow \infty$ implies that
 $\frac{\log \hspace{2 pt} \log \hspace{2 pt} |h^{n+1}(z)|}{n} \rightarrow \infty$ as $n \rightarrow \infty.$ This shows that $T(h) \neq \emptyset$ and $I_0(h) \subset T(h).$\\
We now show that $I_0(h) \subset T(h)$. To this end,
let $z_0 \in I_0(h)$. Then
\[
  \frac{\log |h^{n+1}(z)|}{\log |h^n(z)|} \to \infty.
\]

Put $b_n = \log |h^n(z)|$.
Then $\dfrac{b_{n+1}}{b_n} \to \infty$. We want to show
$\dfrac{\log b_{n+1}}{n} \to \infty$.

As $\dfrac{b_{n+1}}{b_n} \to \infty$, for each $A > 0$, there exists $M$ such that
$\forall\, n \geq M$,
\[
  b_{n+1} \geq L\, b_n.
\]

On iterating $k$ times from $M$ onwards we get
\[
  b_{M+k} \geq L^k\, b_M \quad \forall\, k \geq 0.
\]

Taking $\log$, this gives $\log b_{M+k} \geq k \log L + \log b_M$.

Let $n = M + k$, i.e., $k = n - M$. Then
\[
  \log b_n \geq (n - M)\log L + \log b_M.
\]
\[
  \Rightarrow \quad \frac{\log b_n}{n} \geq \frac{(n - M)}{n}\log L + \frac{\log b_M}{n}
  = \left(1 - \frac{M}{n}\right)\log L + \frac{\log b_M}{n}.
\]

The right hand side tends to $\log L$ as $n \to \infty$.

\[
  \Rightarrow \quad \frac{\log b_n}{n} \geq \log L \quad \text{as } n \to \infty.
\]

As $L$ was arbitrary, we get
\[
  \liminf_{n \to \infty} \frac{\log b_n}{n} \geq \log L \quad \text{for each } L > 0.
\]

On letting $L \to \infty$,
\begin{equation}
  \frac{\log b_n}{n} \to \infty. \tag{1}
\end{equation}

Observe that
\[
  \frac{\log b_{n+1}}{n} = \frac{(n+1)}{n} \cdot \frac{\log b_{n+1}}{n+1} \to \infty
  \quad \text{(using (1))}.
\]

\[
  \Rightarrow \quad \frac{\log b_{n+1}}{n} \to \infty.
\]

Substituting $b_n = \log |h^n(z)|$, we obtain
\[
  \frac{\log \log |h^{n+1}(z)|}{n} \to \infty,
\]
and hence the result.

Using fast escaping set, Bergweiler and Hinkkanen  proved the following theorem regarding equality of Julia sets for permutable functions.}
\begin{theorem}
(\large{\cite{Bergweiler2}, Corollary of Theorem 2)
If $f$ and $g$ are permutable entire functions such that $A(f) \subset J(f)$ and $A(g) \subset J(g),$ then $J(f)=J(g)$. In particular, the conclusion  holds if $f$ and $g$ have no wandering domains. 
}
\end{theorem}
\noindent\large{ Note that if a transcendental entire function $h$ does not have wandering domains, then $A(h)\subset J(h).$The conclusion $J(f)=J(g)$ for two permutable maps $f$ and $g$  having no wandering domains was obtained by Langley \cite{Langley} under additional growth restrictions on $f$ and $g.$}
  
\noindent \large{In 2002, Wang and Yang established some results for Fatou components of permutable functions and showed some progress on escaping sets of entire functions. They took a non-constant polynomial $Q$ and permutable entire functions $f$ and $g$ satisfying the relation $Q(g) = aQ(f) + b, \ a (\neq 0), \ b \in \mathbb{C}.$ The results are discussed in the following theorems.}
\begin{theorem}
\large{(\cite{Wang_Yang}, Theorem 1)
Let $f$ and $g$ be two  permutable  entire functions and
$Q$ be a non-constant polynomial. Suppose that $Q(g) = aQ(f) + b, \ a (\neq 0), \ b \in \mathbb{C}$. Then following are true:
\begin{enumerate}
\item[(i)] {$I_0(f) = I_0(g)$}
\item[(ii)] {$T (f) = T(g)$}
\item[(iii)] {$J (f ) = J (g)$ where $I_0(f), T(f)$ are defined as in \eqref{I_0}, \eqref{T} respectively.}
\end{enumerate} 
}
\end{theorem}

\begin{theorem}
\large({\cite{Wang_Yang}, Theorem 2)\label{th1,}
Let $f$ and $g$ be two distinct permutable  entire functions and
$Q$ be a non-constant polynomial. Suppose that $Q(g) = aQ(f ) + b, \ a (\neq 0), \ b \in \mathbb{C}$. Then
\begin{enumerate}
\item[(i)] {if $g$ has at least one fixpoint, then $A(f) \subset A(g)$}
\item[(ii)] {$A(f^2) = A(g^2)$}
\item[(iii)] {$A(f ) \subset A(g)$ or $A(g) \subset A(f )$} \cite{Zheng_Yang}.
\end{enumerate}
}
\end{theorem}
\noindent\large{ Wang and Yang further conjectured that if $f$ and $g$ satisy the conditions of Theorem \ref{th1,}, then $A(f)=A(g).$ In the case when $f$ and $g$ have a special form, they further obtained:}
\begin{theorem}
\large({\cite{Wang_Yang}, Corollary 1)
Let $f$ and $g$ be two distinct permutable  entire functions of
finite order such that $f$ is of the following form: \\
$f(z) = P_0(z) + P_1(z)e^{Q_1(z)}$ or  $f(z)=P_0(z) + P_1(z)e^{Q_1(z)} + P_2(z)e^{Q_2(z)}$ \\
where $Q_i(z) \ (i = 1, 2)$ and $P_i(z) \ (i = 0, 1, 2)$ are polynomials, $P_0(z)$ is not a constant. Then $J(f)= J(g), \  T(f) = T(g), \ I_0(f ) = I_0(g)$ and $A(f^2) = A(g^2)$.
} 
\end{theorem}
\noindent\large{Here $P_1, P_2$ can be constant. As $f$ is transcendental, so $P_1, Q_1$ cannot be constant simultaneously.}
\begin{theorem}\label{lem1}
\large{(\cite{Wang_Yang}, Lemma 1)
Let $f$ and $g$ be two distinct permutable  entire functions and
$q$ be a non-constant polynomial. Suppose that $q(g) = aq(f) + b, \ a (\neq 0), b \in \mathbb{C}$. Then, for $ n \geq 1$ and $z \in \mathbb{C},$ 
$ q(g^n(z))= a^n q(f^n(z))+ b(a^{n-1}+ a^{n-2}+a+1).$
}
\end{theorem}
\noindent\large {We give an outline of the proof of above Theorem \ref{lem1} as mentioned in \cite{Wang_Yang}.\\
We  prove it by induction. The result clearly holds for $n = 1.$ Assume  $n > 1$ and suppose the assertion is true for $n-1,$ that is, 
$q \left(g^{n-1} (z)\right) = a^{n-1}q(f^{n-1}(z))+ b(a^{n-2} + a^{n-3}+ \cdots  +a+1).$ 
Thus,
$q \left(g^{n-1} \circ f\right) = a^{n-1}q(f^{n})+ b(a^{n-2} + a^{n-3}+ \cdots  +a+1).$
and hence, 
$q \left(f \circ g^{n-1} \right) = a^{n-1}q(f^{n})+ b(a^{n-2} + a^{n-3}+ \cdots  +a+1).$ \\
From the assumption of this lemma, we can get that $(q(g) - b)/a = q(f)$, hence\\
$\frac{q(g^n)-b}{a} = \frac{q(g)-b}{a} \circ g^{n-1}= q(f \circ g^{n-1})= a^{n-1}q(f^{n})+ b(a^{n-2} + a^{n-3}+ \cdots  +a+1).$ \\
Therefore $q(g^n(z))= a^n q(f^{n}(z))+ b(a^{n-1} + a^{n-2}+ \cdots  +a+1).$
}
\\
\noindent\large{ In the case when the polynomial $q(z)=z,$ Wang and Yang further answered a  question of Baker raised in the discussion preceding Theorem \eqref{S_permu}.}
\begin{theorem}
\large({\cite{Wang_Yang}, Theorem 3)
Let $f$ and $g$ be two  permutable  entire functions where
$g = af + b.$\\
Case 1. If $a = 1$ (and $b \neq 0$), then
	\begin{enumerate}
	\item[(i)] {$f(z)$ and $g(z)$ have no multiply connected Fatou components}
	\item[(ii)] {If $D$ is a (pre-)attractive domain, a (pre-)parabolic domain or a (pre-)Siegel disk of $f$ (or of $g$), then $D$ is a wandering domain or a (pre-)Baker domain of $g$ (or of $f$)}	
	\item[(iii)]
	{If $D$ is a (pre-)Baker domain of $f$ , then $D$ is a (pre-)Baker domain or a wandering domain of $g$.}	
	\end{enumerate}
Case 2. If $a \neq 1,$ then the respective Fatou components of $f$ and $g$ belongs to the same classes.
}
\end{theorem}

\noindent \large{In 2016, Benini, Rippon and Stallard in their paper showed some progress on when $J(f)=J(g)$ will be satisfied for permutable  entire functions. They have considered the existence of fast escaping sets and wandering domains. Their results are mentioned in the following theorems:}

\begin{theorem} \label{not_fast_escape}
\large{(\cite{Rippon}, Proposition 3.2)
Let $f$ and $g$ be permutable  entire functions and let $U$ be a Fatou component of $f$ that is not contained in the fast escaping set of $f.$ Then $g(U) = V$, where V is a Fatou component of $f$ that is not contained in the fast escaping set of $f.$ 
}
\end{theorem}
\noindent\large{ Another important contribution of \cite{Rippon} to permutable entire functions is the following result: }
\begin{theorem} \label{multiply_connec_WD}
\large{(\cite{Rippon}, Proposition 3.3)
Let $f$ and $g$ be permutable entire functions and let $U$ be a multiply connected wandering domain of $f$. Then $g(U)$ is a multiply connected wandering domain of $f$.
}
\end{theorem}

\begin{theorem} \label{anna}
\large({\cite{Rippon}, Theorem 3.1)
Let $f$ and $g$ be permutable entire functions. Then 
$F^*(f) \subset F(g)$ and $F^*(g) \subset F(f)$ where $F^*(f),F^*(g)$ represents union of those Fatou components which are not simply connected Fast Escaping wandering domains.
}
\end{theorem}

\noindent \large{Theorem \ref{anna} can be proved using Theorems \ref{not_fast_escape}, \ref{multiply_connec_WD} and considering Montel's theorem. The next theorem  is again an immediate consequence of above theorem.}

\begin{theorem}
\large({\cite{Rippon}, Theorem 1.3)
Let $f$, $g$  be permutable entire functions such that $f$, $g$ do not have simply connected fast escaping wandering domains. Then $J(f)=J(g)$.
}
\end{theorem}

\noindent \large{An important point  to be noted is that any Fatou component contained in $A(f)$ is always a wandering domain. As a result, a Baker domain can never be contained inside the fast escaping set. Also, given a fast escaping wandering domain, it  can either be simply connected or multiply connected. It is known that multiply connected wandering domains are always fast escaping \cite{rip1}. Moreover,  it has been observed that most known simply connected wandering domains are not fast escaping.}

\noindent \large{Now, we  mention some results about periodic entire functions. The following result is by Gross in 1966.

\begin{theorem}
\large{(\cite{Gross}, Theorem 3)
For two  non-linear permutable functions $f$ and $g$,   if their composition $h=f \circ g$ is of finite order, then $h$ is not  periodic.
}
\end{theorem}
\begin{remark}
If we lift the restriction of finite order on $h$ in above theorem, then $h$ could be periodic. For instance, for $f(z)=e^z=g(z),$  $h(z)=e^{e^z}$  is periodic.
\end{remark} 
\noindent \large{By an old result of Polya \cite{polya}, there are only two choices  for $f$ and $g$ whenever $f\circ g$ is of finite order namely,  (i) either $g$ is a polynomial and $f$ is of finite order, or (ii) $g$ is not a polynomial  of finite order and $f$ is of zero order.}

\noindent \large{Now, we see some results for different subsets of Fatou and Julia sets as in how the subsets of the composite function are related to subsets of individual functions, how the points characterised by one function behave with respect to others etc. In this direction, some results have been proven by Baker, for instance, regarding repelling fixed points \cite{Baker_permu2}.}

\noindent \large{Fatou ({\cite{F2}, p. 354) showed the importance of the periodic points in the dynamics of entire function by establishing that every point in the Julia set is a limit point of periodic points. This result of Fatou was further strengthened by Baker \cite{Baker_permu2} who established that each point of Julia set is a limit point of repelling periodic points, in other words, Julia set equals the closure of the set of repelling periodic points.
 Following are some results  regarding the number of fixed points of composite entire functions and their properties. These results were proved by Bergweiler in 1991.}

\begin{theorem} \label{Fixedpt2}
\large({\cite{Bergweiler1}, Theorem 4)
Let $f$ and $g$ be  entire functions. Then $f \circ g$ has infinitely many repelling fixed points. 
}
\end{theorem}

\noindent \large{In 1998, Bergweiler and Wang showed some dynamical properties of compositions of two entire functions taken in opposite orders. We will establish  analogous results for the bungee set, the filled Julia set and the escaping set later on. }
\begin{theorem} \label{Berg_Julia_compo}
\large({\cite{Bergweiler5} Theorem 1)
Let f and g be  entire functions and $z \in \mathbb{C}$. Then $z \in J(f\circ g)$ if and only if $g(z) \in J(g \circ f).$  
}
\end{theorem}
\noindent\large{ Theorem \ref{Berg_Julia_compo} gives that $z$ is in the Fatou set of $f\circ g$ if and only if $g(z)$ is in the Fatou set of $g\circ f.$ The next result relates components of the composite function.}
\begin{theorem}\label{berg'}
\large({\cite{Bergweiler5}, Theorem 2)
Let $f$ and $g$ be  entire functions. Let $U_0$ be a component of $F(f \circ g)$ and let $V_0$ be the component of $F(g \circ f)$ that contains $g(U_0).$ Then 
\begin{enumerate}
\item[(i)] {$U_0$ is wandering if and only if $V_0$ is wandering. }
\item[(ii)] {If $U_0$ is periodic, then so is $V_0.$ Moreover, $V_0$ is of the same type according to the classification of periodic components as $U_0.$} 
\end{enumerate}
}
\end{theorem}
\noindent\large{ Theorem \ref{berg'} part (1) gives that $f\circ g$ has wandering domain if and only if $g\circ f$ has wandering domain.}

\noindent\large{We now consider an alternate  partition of the complex plane by considering the nature of the orbit of a point. By this we mean that we consider three different sets of points namely, $K(f), I(f)$ and $BU(f)$ defined below. The filled Julia set $K(f)$ is defined as $K(f)=\{ z \in \mathbb{C}:f^n(z)$ is bounded for all $n\in\N \}$ \cite{Osb1}}.  $I(f)$ is the escaping set already defined earlier in this section. The complement of $K(f)\cup I(f)$ is known as the bungee set of $f$ denoted by $BU(f)$ which was first defined by Osborne and Sixmith \cite{Osb}.  It is defined as $BU(f)$=$\{z\in\C: \exists$ at least two subsequences ${f^{m_k}}$ and ${f^{n_k}}$ with $m_k,\ n_k \rightarrow \infty$ as $k \rightarrow \infty$ and a constant $R>0$ such that $|{f^{m_k}(z)}|\leq R$ for $k=1,2,3 \ldots, $ and $|{f^{n_k}(z)}| \rightarrow \infty$ as $\ n_k \rightarrow \infty \}$.}

%

\noindent\large{As $K(f)$ and $I(f)$ are always completely invariant, it follows that $BU(f)$ is also completely invariant. An important result connecting Fatou component and the bungee set was proved in \cite{Osb}.}
\begin{theorem}\label{thm osb}
\large({\cite {Osb}, Theorem 1.1)
Let $g$ be an entire function such that $U\cap BU(g)\not=\emptyset$, where $U\subset F(g)$ is a Fatou component. Then 
\begin{enumerate}
\item [(i)] $U\subset BU(g)$ and U is a wandering domain of $g$;
\item [(ii)] $J(g)=\partial BU(g)$.
\end{enumerate}
} 
\end{theorem}
\noindent\large{The following result is an easy observation that deals with the pre-images of the bungee sets corresponding to commuting entire functions.}
\begin{lemma}
\large({\cite{AP}, Theorem 2.4)
Let $f$ and $g$ be permutable entire functions. Let $U \subset BU(f).$ If $g^{-1}(U)\neq \emptyset$, then $g^{-1}(U)\cap(I(f)\cup BU(f)) \neq \emptyset$.
}
\end{lemma}
Note that above result always  holds if $|U|\geq 2.$

\noindent \large{In  \cite{dk}, the relation between escaping sets of individual functions, their compositions was provided. So far we have observed that the Julia set of two permutable entire functions may not always coincide. However, the following is always true for two permutable entire functions.}

\begin{theorem} \label{Fatou_set_composition}
\large({\cite{dk} Lemma 3.2)
Let $f$ and $g$ be permutable  entire functions. Then $F(f \circ g) \subset F(f) \cap F(g)$.
}
\end{theorem}
We reproduce the proof from \cite{dk}.
\begin{proof}
In Theorem \ref{Berg_Julia_compo}, it was shown that $z\in F(f\circ g)$ if and only if $f(z)\in F(g\circ f).$ As $f\circ g=g\circ f,\, F(f\circ g)$ is completely invariant under $f$ and by symmetry, under $g$ respectively and so, in particular, it is forward invariant under them. So $f(F(f\circ g))\subset F(f\circ g)$ and $g(F(f\circ g))\subset F(f\circ g),$  which by Montel's Normality Criterion implies $F(f\circ g)\subset F(f)$ and $F(f\circ g)\subset F(g)$ and hence the result. 
\end{proof}
\noindent \large{As Julia set is the complement of Fatou set, using Theorem \ref{Fatou_set_composition} and  on taking complement of it, we obtain the result $J(f \circ g) \supset J(f) \cup J(g).$ 
Recall that for a function $f$ in the Eremenko-Lyubich class $\mathfrak B, I(f)\subset J(f)$ and $ J(f)=\overline{I(f)}$ \cite{EL}. As a result, we obtain for permutable $f$ and $g$  in the Eremenko-Lyubich class $\mathfrak B, \overline{I(f)} \cup \overline{I(g)} \subset \overline{I(f \circ g)}.$}


\begin{theorem} \label{Compo_I}
\large({\cite{dk} Theorem 3.4, Theorem 3.5)
Let $f$ and $g$ be permutable  entire functions. Then,
\begin{enumerate}
\item {$g^{-1}(I(f)) \subset I(f).$}
\item{$I(f \circ g)$ is completely invariant under $f$ and $g$ respectively.}
\item{$I(f \circ g) \subset I(f) \cup I(g)$}
\item{ For $i\neq j \in \mathbb{N}, \ I(f^{i}\circ g^j)=I(f \circ g)$}.

\end{enumerate} 
}
\end{theorem}
We reproduce the proof for completeness.

\begin{proof}
\begin{enumerate}
\item {This proof is straight forward.}
\item {We first show that $z\in I(f\circ g)$ if and only if $g(z)\in I(g\circ f).$ Let $z\in I(f\circ g).$ Then $(f\circ g)^n(z)\to\ity$ as $n\to\ity,$ that is,
$f((g\circ f)^{n-1}g(z))\to\ity$ as $n\to\ity.$ As $f$ is an entire function, this implies that $(g\circ f)^{n-1}g(z)\to\ity$ as $n\to\ity,$ that is, $g(z)\in I(g\circ f).$ On the other hand, let $g(z)\in I(g\circ f).$ Then $(g\circ f)^{n}(g(z))\to\ity$ as $n\to\ity,$ that is, $g((f\circ g)^n(z))\to\ity$ as $n\to\ity.$ Again, as $g$ is entire, this forces $(f\circ g)^n(z)\to\ity$ as $n\to\ity.$  So, $z\in I(f\circ g)$ which proves the claim. As $f\circ g=g\circ f,$ we obtain $z\in I(f\circ g)$ if and only if $g(z)\in I(f\circ g)$ which implies $I(f\circ g)$ is completely invariant under $g,$ and by symmetry, under $f$ respectively.}
\item {Suppose that \(z\not\in I(f)\) and \(z\not \in I(g)\). We shall prove that \(z\not \in I(f\circ g)\). \\Consider the following infinite sequences of rows where each row consists of orbit of $\{g^k(z): k\in\N\}$ under $f:$

\[\begin{array}{ll}
\vspace{3mm}
 f(g(z)) , f^2(g(z)) ,f^3(g(z)),\ldots\\
\vspace{3mm} 
f(g^2(z)) , f^2(g^2(z)) ,f^3(g^2(z)),\ldots\\
\vspace{3mm}
 f(g^3(z)) , f^2(g^3(z)) ,f^3(g^3(z)),\ldots\\
\vspace{3mm}
  \quad \vdots \quad\quad\quad \quad \quad \vdots\quad \quad \quad \quad  \quad \vdots\\
f(g^k(z)) ,f^2(g^k(z)) ,f^3(g^k(z)),\ldots\\
  \quad \vdots \quad\quad\quad \quad \quad \vdots\quad \quad \quad \quad  \quad \vdots
     \end{array}\]


 Since \(z\not\in I(f)\) and \(z\not \in I(g)\), therefore each of the above sequence is bounded and hence has a limit point. We can assume that each of the above sequence is actually convergent. Let \(l\) be the limit  of the sequence corresponding to first row, i.e., \(f^n(g(z))\to l\) as \(n\to \infty\) which gives \( f^n(g^2(z)))\to g(l)\) as \(n\to \infty\). Proceeding on similar lines, we have  \(f^n(g^k(z)))\to g^{k-1}(l)\) as \(n\to \infty.\) Let 
\[X=\{l,g(l),g^2(l)\ldots\}.\]
 We have the following cases to consider.\\ 
Case 1: \(l= g(l)\) i.e., \(l\) is a fixed point of \(g\), then \(g^k(l)=l , \forall\,  k \in\N.\)   In particular, we can extract a subsequence \(\{f^n(g^n(z)), n\in \N\}\) which also converges to \(l\),  i.e., \(\lim\limits_{n\to \infty}f^n(g^n(z))=l(\neq \infty)\) from which we see that \( z\not \in I(f\circ g)\).
\\Case 2: \(g^m(l)=g^{m+j}(l)\) for some  \(m\in \mathbf{N}\) and $j\geq 1$ i.e., \(g^m(l)\) is a fixed point of \(g\), which in turn implies
\begin{equation}\notag
\begin{split}
 \lim\limits_{n\to \infty}f^n(g^{m+j}(z))
&=g^{m+j}(l)\\
&=g^m(l), j\geq 1.
\end{split}
\end{equation}
This gives  \( \lim\limits_{n\to\infty}f^n(g^k(z))=g^m(l), \forall\, k\geq m\). Therefore, \(\lim\limits_{n\to \infty}f^n(g^n(z))=g^m(l)(\neq \infty)\) from which we deduce that  \(z\not \in I(f\circ g)\).\\ 
Case 3: When \(X\) is an infinite set then it  has a limit point \(\alpha\, \mathit{(say)}\)  
     i.e., \(g^{k-1}(l)\to \alpha\) as \(k\to \infty\). We consider the following two subcases:\\
     Subcase 1: \(\alpha\) is finite.\\
     As we have \(f^n(g(z))\to l\) as \(n\to \infty\), so by continuity of \(g\) we have
     \(\lim\limits_{n\to \infty}f^n(g^k(z))= g^{k-1}(l)\). This implies that  \( \lim\limits_{k\to \infty}\lim\limits_{n\to \infty}f^n(g^k(z))= \alpha\). Thus, every subsequence  will also converge to \(\alpha\), i.e., \(\lim\limits_{n\to \infty}f^n(g^n(z))=\alpha\) so that \( z\not\in I(f\circ g).\)\\
     Subcase 2: \(\alpha\) is infinity.\\
     By using  similar arguments, we have \(\lim\limits_{k\to\infty}\lim\limits_{n\to \infty}f^n(g^k(z))=\infty\). In particular, \(\lim\limits_{n\to\infty}f^n(g^k(z))=\infty\), i.e. \(g^k(z)\in I(f)\) which implies that \(z\in I(f) \), which is a contradiction to our supposition. 
This completes the proof of the result.}

\item {For $i\neq j\in\N,$ assume $i\geq j.$ Now using previous result we have
\begin{equation}\notag
\begin{split}
 I(f^i\circ g^j)
&=I(f^{i-j}\circ {f^j\circ g^j})\\
&\subset I(f^{i-j})\cup I(f^j\circ g^j)\\
&=I(f)\cup I(f\circ g)\\
&\subset I(f)\cup I(g).\qedhere
\end{split}
\end{equation}}

\end{enumerate}
\end{proof}
\noindent \large{Recently,  some more results on the escaping set, the bungee set and the filled Julia set  were obtained in \cite{D_R}. For completeness, we reproduce the proofs of some of these results.}
%
%
%
%


\begin{theorem}\label{DR1}
\large({ \cite{D_R} Theorem 2.3})
Suppose $f$, $g$ are two  entire functions. Then
\begin{enumerate}
\item  $z_0 \in BU(f \circ g)$ if and only if  $g(z_0) \in BU(g \circ f);$ 
\item $z_0 \in K(f \circ g)$ if and only  if $g(z_0) \in K(g \circ f);$
\item $z_0 \in I(f \circ g)$ if and only  if $g(z_0) \in I(g \circ f).$
\end{enumerate}
\end{theorem}
\begin{proof}
It is enough to provide the proof of the first part of the theorem as the proofs of the other two parts follows on the similar lines.\\
We first assume that $z_0 \in BU(f \circ g)$ and then show that $g(z_0) \in BU(g \circ f).$ On the contrary, assume $g(z_0)\in I(g\circ f)\cup K(g\circ f).$ If $g(z_0)\in I(g\circ f),$ then  $(g\circ f)^{n}(g(z_0))\to\infty$ as $n\to\infty.$ This implies that $(g(f\circ g)^{n}(z_0))\to\infty$ as $n\to\infty$. As $g$ is entire we obtain $z_0\in I(f\circ g)$ which is a contradiction. On the other hand, if  $g(z_0)\in K(g\circ f)$, then there exists a constant $C>0$ such that $|(g\circ f)^n(g(z_0))|\leq C \mbox{ for all } n\in\N$. By applying $f$ we obtain, $|(f\circ g)^{n+1} (z_0)|\leq M(C, f) \mbox{ for all } n\in\N$, which implies that $z_0\in K(f\circ g)$ which is again a contradiction. Hence, we conclude that $g(z_0) \in BU(g \circ f).$\\
Conversely, assume that $g(z_0) \in BU(g \circ f)$ and then show that $z_0 \in BU(f \circ g).$ 
 Suppose on the contrary, $z_0\in I(f\circ g)\cup K(f\circ g).$ 
 If $z_0\in I(f\circ g) $, then $(f\circ g)^{n}(z_0)\to\infty$ as $n\to\infty.$ This implies that $(f(g\circ f)^{n-1}(g(z_0)))\to\infty$ as $n\to\infty$. As $f$ is entire, we obtain $g(z_0)\in I(g\circ f)$ which is a contradiction. On the other hand, if  $z_0\in K(f\circ g)$, then there exists a constant $B>0$ such that $|(f\circ g)^n(z_0)|\leq B \mbox{ for all } n\in\N$. This further implies that $|(g\circ f)^n (g(z_0))|\leq M(B, g) \mbox{ for all } n\in\N$, which implies that $g(z_0)\in K(g\circ f)$ which is again a contradiction. On combining both the observations, we conclude that  $z_0\in BU(f\circ g)$.
\end{proof}
\noindent\large{This is an analogous result to Theorem \eqref{Berg_Julia_compo} proved by Bergweiler and Wang in 1998.
\noindent\large{We finally conclude the following from above result:
\begin{theorem}\label{DR4}
For two  entire functions $f$ and $g,$ the following holds:
\begin{enumerate}
\item $g(BU(f\circ g)) = BU(g\circ f);$\\
\item $g(K(f\circ g)) = K(g\circ f);$\\
\item $g(I(f\circ g)) = I(g\circ f).$
\end{enumerate}
\end{theorem}

As a special case, we obtain the following result.
\begin{corollary}
Suppose $f$ and $g$ are two permutable  entire functions. Then $BU(f\circ g), K(f\circ g)$ and $I(f\circ g)$ are completely invariant under both $f$ and $g.$
\end{corollary}

%

\bigskip


\end{document}